\documentclass{amsart}
\usepackage{amssymb,amsmath,amsthm,enumerate,amscd,graphicx,color}
\usepackage{pdfsync}
\numberwithin{equation}{section}
\newtheorem{thm}{Theorem}[subsection]
\newtheorem{lem}[thm]{Lemma}

\newtheorem{prop}[thm]{Proposition}

\newcommand{\propref}[1]{Proposition~\ref{#1}}

\newcommand{\eqnref}[1]{~(\ref{#1})}

\begin{document}



\title{}

\title{ DJKM algebras I: Their universal central extension.}
\author{ Ben Cox}
\author{Vyacheslav Futorny}
\keywords{Krichever-Novikov Algebras,  Landau-Lifshitz differential equation, Date-Jimbo-Miwa-Kashiwara algebras, universal central extension, ultraspherical polynomials, elliptic integrals}
\address{Department of Mathematics \\
University of Charleston \\
66 George St.  \\
Charleston, SC 29424, USA}\email{coxbl@cofc.edu}
\address{Department of Mathematics\\
 University of S\~ao Paulo\\
 S\~ao Paulo, Brazil}
 \email{futorny@ime.usp.br}
 \begin{abstract} The purpose of this paper is to explicitly describe in terms of generators and relations the universal central extension of the infinite dimensional Lie algebra, $\mathfrak g\otimes \mathbb C[t,t^{-1},u|u^2=(t^2-b^2)(t^2-c^2)]$, appearing in the work of Date, Jimbo, Kashiwara and Miwa in their study of integrable systems arising from Landau-Lifshitz differential equation.  
 \end{abstract}
\date{}
\thanks{
 The second author were partially supported by Fapesp (processo 2005/60337-2) and CNPq (processo 301743/2007-0). The first author is grateful to the Fapesp (processo 2009/17533-6)  and the University of Sao Paulo for the support and hospitality during his visit to Sao Paulo. The first author was also partially supported by a research and development grant from the College of Charleston.}


\maketitle

\section{Introduction} 
In this paper the authors explicitly describe in terms of generators and relations and three families of polynomials, the universal central extension of an algebra appearing in work of Date, Jimbo, Kashiwara and Miwa (see \cite{MR701334,MR823315}), where they study integrable systems arising from Landau-Lifshitz differential equation.  Two of these families of polynomials are described below in terms of elliptic integrals and the other family is a variant of certain ultraspherical polynomials.  The  authors Date, Jimbo, Kashiwara and Miwa solved the Landau-Lifshitz equation using methods developed in some of their previous work on affine Lie algebras.  The hierarchy of this equation is written in terms of free fermions on an elliptic curve. The  infinite-dimensional Lie algebra mentioned above is shown to act on solutions of the Landau-Lifshitz equation as infinitesimal B\"acklund transformations where they derive an $N$-soliton formula.These authors arrive at an algebra that is a one dimensional central extension of $\mathfrak g\otimes \mathbb C[t,t^{-1},u|u^2=(t^2-b^2)(t^2-c^2)]$ where $b\neq \pm c$ are complex constants and $\mathfrak g$ is a simple finite dimensional Lie algebra defined over the complex numbers.  Below we explicitly describe its four dimensional universal central extension.  Modulo the center this algebra is a particular example of a Krichever-Novikov current algebra (see (\cite{MR902293}, \cite{MR925072}, \cite{MR998426}).  A fair amount of interesting and fundamental work has be done by Krichever, Novikov, Schlichenmaier, and Sheinman on the representation theory of certain one dimensional central extensions of these latter current algebras and of analogues of the Virasoso algebra. In particular Wess-Zumino-Witten-Novikov theory and analogues of the Knizhnik-Zamolodchikov equations are developed for  these algebras (see the survey article \cite{MR2152962}, and for example \cite{MR1706819}, \cite{MR1706819},\cite{MR2072650},\cite{MR2058804},\cite{MR1989644}, and \cite{MR1666274}).

M. Bremner on the other hand has explicitly described in terms of generators, relations and certain families of polynomials (ultraspherical and Pollacyk) the structure constants for the universal central extension of algebras of the form $\mathfrak g\otimes \mathbb C[t,t^{-1},u|u^2=p(t)]$ where $p(t)=t^2-2bt+1$ and $p(t)=t^3-2bt^2+t$ (see \cite{MR1249871,MR1303073}).  He determined more generally the dimension of the universal central extension for affine Lie algebras of the form $\mathfrak g\otimes R$ where $R$
is the ring of regular functions defined on an algebraic curve with any number of points removed. He obtained this using C.  Kassel's result (\cite{MR772062}) where one knows  that the center is  isomorphic  as a vector space to $\Omega_R^1/dR$ (the {\it space of K\"ahler  differentials  of $R$ modulo  exact forms}).    We will review this material below as needed. 

In previous work of the authors (see  \cite{MR2373448,MR2541818}) we used Bremner's aforementioned description to obtain certain free field realizations of the four point and elliptic affine algebras depending on a parameter $r=0,1$ which correspond to two different normal orderings.  These later realizations are analogues of Wakimoto type realizations which have been used by Schechtman and Varchenko and various other authors in the affine setting to pin down integral solutions to the Knizhnik-Zamolodchikov differential equations (see for example \cite{MR1136712}, \cite{MR1138049}, \cite{MR2001b:32028}, \cite{MR1077959}).   Such realizations have also been used in the study of Drinfeld-Sokolov reduction in the setting of $W$-algebras and in E. Frenkel's and B. Feigin's description of the center of the completed enveloping algebra of an affine Lie algebra (see \cite{MR1309540}, \cite{MR2146349}, and \cite{MR1187549}).   In future work we plan to use results of this paper to describe free field realizations of the universal central extension of the algebras of Date, Jimbo, Kashiwara and Miwa (which, since this is a mouth full, we will call DJKM algebras).

\section{Universal Central Extensions of Current Algebras}  Let $R$ be a commutative algebra defined over $\mathbb C$.
Consider the left $R$-module  with  action $f( g\otimes h ) = f g\otimes h$ for $f,g,h\in R$ and let $K$  be the submodule generated by the elements $1\otimes fg  -f \otimes g -g\otimes f$.   Then $\Omega_R^1=F/K$ is the module of K\"ahler differentials.  The element $f\otimes g+K$ is traditionally denoted by $fdg$.  The canonical map $d:R\to \Omega_R^1$ by $df  = 1\otimes f  + K$.  The {\it exact differentials} are the elements of the subspace $dR$.  The coset  of $fdg$  modulo $dR$ is denoted by $\overline{fdg}$.  As C. Kassel showed the universal central extension of the current algebra $\mathfrak g\otimes R$ where $\mathfrak g$ is a simple finite dimensional Lie algebra defined over $\mathbb C$, is the vector space $\hat{\mathfrak g}=(\mathfrak g\otimes R)\oplus \Omega_R^1/dR$ with Lie bracket given by
$$
[x\otimes f,Y\otimes g]=[xy]\otimes fg+(x,y)\overline{fdg},  [x\otimes f,\omega]=0,  [\omega,\omega']=0,
$$
  where $x,y\in\mathfrak g$, and $\omega,\omega'\in \Omega_R^1/dR$ and $(x,y)$  denotes the Killing  form  on $\mathfrak g$.

Consider the polynomial
$$
p(t)=t^n+a_{n-1}t^{n-1}+\cdots+a_0
$$
where $a_i\in\mathbb C$ and $a_n=1$. 
Fundamental to the description of the universal central extension for $R=\mathbb C[t,t^{-1},u|u^2=p(t)]$ is the following:
\begin{thm}[\cite{MR1303073},Theorem 3.4]  Let $R$ be as above.  The set 
$$
\{\overline{t^{-1}\,dt},\overline{t^{-1}u\,dt},\dots, \overline{t^{-n}u\,dt}\}
$$
 forms a basis of $\Omega_R^1/dR$ (omitting $\overline{t^{-n}u\ dt}$ if $a_0=0$).   
\end{thm}
Set $u^m=p(t)$.  Then $u\,d(u^m)=mu^mdu$ and
$$
\sum_{j=1}^nja_jt^{j-1}u\,dt-m\left(\sum_{j=0}^na_jt^j\,du\right)=0
$$
or 
$$
p'(t)udt-mp(t)du=0.
$$
Multiplying by $t^i$ we get 
\begin{equation}\label{lady1}
\sum_{j=1}^nja_jt^{i+j-1}u\,dt-m\left(\sum_{j=0}^na_jt^{i+j}\,du\right)=0
\end{equation}

\begin{lem}  If $u^m=p(t)$ and $R=\mathbb C[t,t^{-1},u|u^m=p(t)]$, then in $\Omega_R^1/dR$, one has
\begin{equation}\label{recursionreln}
((m+1)n+im)t^{n+i-1}u\,dt \equiv - \sum_{j=0}^{n-1}((m+1)j+mi)a_jt^{i+j-1}u\,dt\mod dR
\end{equation}
\end{lem}
\begin{proof}
We have expanding $d(t^{i+j}u)$
$$
(i+j)t^{i+j-1}u\,dt\equiv-t^{i+j}\,du\mod dR.
$$
so that \eqnref{lady1} implies 
\begin{equation}
\sum_{j=0}^nja_jt^{i+j-1}u\,dt+m\left(\sum_{j=0}^n(i+j)a_jt^{i+j-1}u\,dt\right)=0 \mod dR
\end{equation}
or
\begin{equation}
\sum_{j=0}^n((m+1)j+mi)a_jt^{i+j-1}u\,dt\equiv 0\mod dR
\end{equation}
This gives \eqnref{recursionreln}.

\end{proof}

\section{Description of the universal central extension of Date-Jimbo-Miwa-Kashiwara algebras}
In the Date-Jimbo-Miwa-Kashiwara  setting one takes $m=2$ and $p(t)=(t^2-a^2)(t^2-b^2)=t^4-(a^2+b^2)t^2+(ab)^2$ with $a\neq \pm b$ and neither $a$ nor $b$ is zero.  We fix from here onward $R=\mathbb C[t,t^{-1},u\,|\,u^2= (t^2-a^2)(t^2-b^2)]$.  As in this case $a_0=(ab)^2$, $a_1=0$, $a_2=-(a^2+b^2)$, $a_3=0$ and $a_4=1$, then letting $k=i-2$ the recursion relation in \eqnref{recursionreln} looks like
\begin{align*}
(6+2k)\overline{t^{k}u\,dt}   
&=-2(k-3)(ab)^2\overline{t^{k-4}u\,dt} +2k(a^2+b^2)\overline{t^{k-2}u\,dt}.
\end{align*}
After a change of variables we may assume that $a^2b^2=1$.   Then the recursion relation looks like
\begin{equation}\label{recursionreln1}
(6+2k)\overline{t^{k}u\,dt}   
=-2(k-3)\overline{t^{k-4}u\,dt} +4kc\overline{t^{k-2}u\,dt},
\end{equation}
after setting $c=(a^2+b^2)/2$, so that $p(t)=t^4-2ct^2+1$.  Let $P_k:=P_k(c)$ be the polynomial in $c$ satisfy the recursion relation 
$$
(6+2k)P_k(c)   
=4k cP_{k-2}(c)-2(k-3)P_{k-4}(c)
$$
for $k\geq 0$.
Then set
$$
P(c,z):=\sum_{k\geq -4}P_k(c)z^{k+4}=\sum_{k\geq 0}P_{k-4}(c)z^{k}.
$$
so that after some straightforward rearrangement of terms we have
\begin{align*}
0&=\sum_{k\geq 0}(6+2k)P_k(c)z^k 
-4c\sum_{k\geq 0}kP_{k-2}(c)z^{k} +2\sum_{k\geq 0}(k-3)P_{k-4}(c)z^{k}  \\
&=(-2z^{-4} +8cz^{-2}-6)P(c,z) +(2z^{-3}-4cz^{-1}+2z)\frac{d}{dz}P(c,z)  \\
&\quad+(2z^{-4}-8cz^{-2})P_{-4}(c)  -4cP_{-3}(c)  z^{-1} -2P_{-2}(c)z^{-2} -4P_{-1}(c)z^{-1}.
\end{align*}
We then multiply the above through by $z^{4}$ to get 
\begin{align*}
0&=(-2+8cz^{2}-6z^4)P(c,z) +(2z-4cz^{3}+2z^5)\frac{d}{dz}P(c,z)  \\
&\quad+(2-8cz^{2})P_{-4}(c)  -4cP_{-3}(c)  z^{3} -2P_{-2}(c)z^{2} -4P_{-1}(c)z^{3}.
\end{align*}

Hence $P(c,z)$ must satisfy the differential equation
\begin{equation}\label{funde}
\frac{d}{dz}P(c,z)-\frac{3z^4-4c z^2+1}{z^5-2cz^3+z}P(c,z)=\frac{2\left(P_{-1}+cP_{-3} \right)z^3 +P_{-2} z^2+(4cz^2-1)P_{-4} }{z^5-2cz^3+z}
\end{equation}
This has integrating factor
\begin{align*}
\mu(z)&
=\exp \int\left( \frac{-2 \left(z^3-cz\right)}{1-2 c z^2+z^4 }-\frac{1}{z}\right)\,dz  \\
&=\exp(-\frac{1}{2} \ln(1-2 c z^2+z^4)-\ln (z))=\frac{1}{z\sqrt{1-2 c z^2+z^4}}.
\end{align*}

\subsection{Elliptic Case 1}
If we take initial conditions $P_{-3}(c)=P_{-2}(c)=P_{-1}(c)=0$ and $P_{-4}(c)=1$ then we arrive at a generating function 
$$
P_{-4}(c,z):=\sum_{k\geq -4}P_{-4,k}(c)z^{k+4}=\sum_{k\geq 0}P_{-4,k-4}(c)z^{k},
$$
defined in terms of an elliptic integral
\begin{align*}
P_{-4}(c,z)&=z\sqrt{1-2 c z^2+z^4}\int \frac{4cz^2-1}{z^2(z^4-2c z^2+1)^{3/2}}\, dz.
\end{align*}
 One way to interpret the right hand integral is to expand $(z^4-2c z^2+1)^{-3/2}$ as a Talyor series about $z=0$ and then formally integrate term by term and multiply the result by the Taylor series of $z\sqrt{1-2 c z^2+z^4}$.    More precisely one integrates formally with zero constant term
 $$
 \int (4c-z^{-2})\sum_{n=0}^\infty Q_n^{(3/2)}(c)z^{2n}\,dz =\sum_{n=0}^\infty \frac{4cQ_n^{(3/2)}(c)}{2n+1}z^{2n+1} -\sum_{n=0}^\infty \frac{Q_n^{(3/2)}(c)}{2n-1}z^{2n-1}
 $$ 
 where $Q_n^{(\lambda)}(c)$ is the $n$-th Gegenbauer polynomial.
After multiplying this by 
$$
z\sqrt{1-2cz^2+z^4}=\sum_{n=0}^\infty Q_n^{(-1/2)}(c)z^{2n+1}
$$
one arrives at the series $P_{-4}(c,z)$.

\subsection{Elliptic Case 2}
If we take initial conditions $P_{-4}(c)=P_{-3}(c)=P_{-1}(c)=0$ and $P_{-2}(c)=1$ then we arrive at a generating function defined in terms of another elliptic integral:
\begin{align*}
P_{-2}(c,z)&=z\sqrt{1-2 c z^2+z^4}\int \frac{1}{ (z^4-2c z^2+1)^{3/2}}\, dz.
\end{align*}

\subsection{Gegenbauer  Case 3}
If we take $P_{-1}(c)=1$, and $P_{-2}(c)=P_{-3}(c)=P_{-4}(c)=0$ and set 
$$
P_{-1}(c,z)=\sum_{n\geq 0}P_{-1,n-4}z^n,
$$
then we get a solution which after solving for the integration constant can be turned into a power series solution 
\begin{align*}
P_{-1}(c,z)&=(z\sqrt{1-2 c z^2+z^4})\left(\int \frac{2cz^3}{t\sqrt{1-2 c z^2+z^4}(z^5-2c z^3+z)}\, dt+C\right)  \\
&=\frac{ z(c-z^3)}{c^2-1}-\frac{c}{c^2-1}z\sqrt{z^4-2cz^2+1} \\
&=\frac{1}{c^2-1}\left(cz-z^3-cz\sqrt{z^4-2c z^2+1}\right)  \\
&=\frac{1}{c^2-1}\left(cz-z^3-\sum_{k=0}^\infty c Q_n^{(-1/2)}(c)z^{2n+1}\right)  \\
&=\frac{1}{c^2-1}\left(cz-z^3-cz+c^2z^3-\sum_{k=2}^\infty c Q_n^{(-1/2)}(c)z^{2n+1}\right)  
\end{align*}
where $Q^{(-1/2)}_n(c)$ is the $n$-th Gegenbauer polynomial.   Hence
\begin{align*}
P_{-1,-4}(c)&=P_{-1,-3}(c)=P_{-1,-2}(c) =P_{-1,2m}(c)=0, \\
P_{-1,-1}(c)&=1,  \\
P_{-1,2n-3}(c)&=\frac{-cQ_{n}(c)}{c^2-1},
\end{align*}
for $m\geq 0$ and $n\geq 2$ .
The $Q^{(-1/2)}_n(c)$ are known to satisfy the second order differential equation:
\begin{align*}
(1-c^2)\frac{d^2}{d^2 c}Q^{(-1/2)}_n(c)+n(n-1)Q^{(-1/2)}_{n}(c)=0
\end{align*}
so that the $P_{-1,k}:=P_{-1,k}(c)$ satisfy the second order differential equation
\begin{align*}
(c^4-c^2)\frac{d^2}{d^2 c}P_{-1,2n-3}+2c(c^2+1)\frac{d}{d  c}P_{-1,2n-3}+(-c^2n(n-1)-2)P_{-1,2n-3}=0
\end{align*}
for $n\geq 2$.

\subsection{Gegenbauer Case 4}
Next we consider the initial conditions $P_{-1}(c)=0=P_{-2}(c)=P_{-4}(c)=0$ with $P_{-3}(c)=1$ and set 
$$
P_{-3}(c,z)=\sum_{n\geq 0}P_{-3,n-4}(c)z^n,
$$
then we get a power series solution
\begin{align*}
P_{-3}(c,z)&=(z\sqrt{1-2 c z^2+z^4})\left(\int \frac{2cz^3}{z\sqrt{1-2 c z^2+z^4}(z^5-2c z^3+z)}\, dz+C\right)  \\
&=\frac{ cz(c-z^3)}{c^2-1}-\frac{1}{c^2-1}z\sqrt{z^4-2cz^2+1} \\
&=\frac{1}{c^2-1}\left(c^2z-cz^3-z\sqrt{z^4-2c z^2+1}\right)  \\
&=\frac{1}{c^2-1}\left(c^2z-cz^3-\sum_{k=0}^\infty Q_n^{(-1/2)}(c)z^{2n+1}\right)  \\
&=\frac{1}{c^2-1}\left(c^2z-cz^3-z+cz^3-\sum_{k=2}^\infty Q_n^{(-1/2)}(c)z^{2n+1}\right)  \\
\end{align*}
where $Q^{(-1/2)}_n(c)$ is the $n$-th Gegenbauer polynomial.  Hence
\begin{align*}
P_{-3,-4}(c)&=P_{-3,-2}(c)=P_{-3,-1}(c) =P_{-1,2m}(c)=0, \\
P_{-3,-3}(c)&=1,  \\
P_{-3,2n-3}(c)&=\frac{-Q_{n}(c)}{c^2-1},
\end{align*}
for $m\geq 0$ and $n\geq 2$ and hence \begin{align*}
(c^2-1)\frac{d^2}{d^2 c}P_{-3,2n-3}+4c \frac{d}{d c}P_{-3,2n-3} -(n+1)(n-2)P_{-3,2n-3} =0
\end{align*}
for $n\geq 2$ and $P_{-1,2n-3}=cP_{-3,2n-3}$ for $n\geq 2$.
\section{Main result}
First we give an explicit description of the cocyles contributing to the {\it even} part of the DJKM algebra. 

\begin{prop}[cf. \cite{MR1303073}, Prop. 4.2] \label{cocyclecalc} Set $\omega_0=\overline{t^{-1}\,dt}$. 
 For $i,j\in\mathbb Z$ one has
\begin{equation}
t^i\,d(t^j)= j \delta_{i+j,0}\omega_0
\end{equation}
and 
\begin{equation}
t^{i-1}u\,d(t^{j-1}u)=\left(\delta_{i+j,-2}(j+1) -2cj\delta_{i+j,0} +(j-1)\delta_{i+j,2}\right)\omega_0.
\end{equation}

\end{prop}
\begin{proof}  First observe that 
$2u\,du=d(u^2)=(4t^3-4c t)\,dt$.
 The second congruence then follows from
\begin{align*}
t^{i-1}u\,d(t^{j-1}u)&=(j-1)t^{i+j-3}u^2\,dt+t^{i+j-2}u\,du \\
&=(j-1)t^{i+j-3}(t^4-2ct^2+1)\,dt+2t^{i+j-2}(t^3-c t)\,dt \\
&=(j-1)(t^{i+j+1}-2ct^{i+j-1}+t^{i+j-3})\,dt+2(t^{i+j+1}-c t^{i+j-1})\,dt \\
&=(j+1)t^{i+j+1}\,dt -2cjt^{i+j-1}\,dt +(j-1)t^{i+j-3}\,dt. 
\end{align*}
\end{proof}

The map $\sigma:R\to R$ given by $\sigma(t)=t^{-1}$, $\sigma(u)=t^{-2}u$ is an algebra automorphism as 
$\sigma(u^2)=t^{-4}u^2=1-2ct^{-2}+t^{-4}=\sigma(1-2ct^{2}+t^{4})$.  This descends to a linear map $\sigma:\Omega_R^1/dR$ where
\begin{align*}
\sigma(\overline{t^{-1}\,dt})&=-\overline{t^{-1}\,dt},\\
\sigma(\overline{t^{-1}u\,dt)}&=\overline{t(t^{-2}u)d(t^{-1})}=-\overline{t^{-3}u\,dt},  \\
\sigma(\overline{t^{-2}u\,dt)}&=\overline{t^2(t^{-2}u)d(t^{-1})}=-\overline{t^{-2}u\,dt}, \\
\sigma(\overline{t^{-3}u\,dt)}&=-\overline{t^{-1}u\,dt}, \\ 
\sigma(\overline{t^{-4}u\,dt)}&=\overline{t^4(t^{-2}u)d(t^{-1})}=-\overline{u\,dt}=-\overline{t^{-4}u\,dt },
\end{align*} 
whereby the last identity follows from the recursion relation \eqnref{recursionreln1} with $k=0$.  Setting $\omega_{-k}=\overline{t^{-k}u\,dt}$, $k=1,2,3,4$, then $\sigma(\omega_{-1})=-\omega_{-3}$,  and $\sigma(\omega_{-l})=-\omega_{-l}$ for $l=2,4$.

\begin{thm} Let $\mathfrak g$ be a simple finite dimensional Lie algebra over the complex numbers with   the Killing form $(\,|\,)$ and define $\psi_{ij}(c)\in\Omega_R^1/dR$ by
\begin{equation}
\psi_{ij}(c)=\begin{cases} 
\omega_{i+j}&\quad \text{ for }\quad i+j=1,0,-1,-2 \\
P_{-3,i+j-2}(c) (\omega_{-3}+c\omega_{-1})&\quad \text{for} \quad i+j =2n-1\geq 3,\enspace n\in\mathbb Z, \\
P_{-3,i+j-2}(c) (c\omega_{-3}+\omega_{-1})&\quad \text{for} \quad i+j =-2n+1\leq - 3, n\in\mathbb Z, \\
P_{-4,|i+j|-2}(c) \omega_{-4} +P_{-2,|i+j|-2}(c)\omega_{-2}&\quad\text{for}\quad |i+j| =2n \geq 2, n\in\mathbb Z. \\
\end{cases}
\end{equation}
The universal central extension of the Date-Jimbo-Kashiwara-Miwa  algebra is the $\mathbb Z_2$-graded Lie algebra 
$$
\widehat{\mathfrak g}=\widehat{\mathfrak g}^0\oplus \widehat{\mathfrak g}^1,
$$
where
$$
\widehat{\mathfrak g}^0=\left(\mathfrak g\otimes \mathbb C[t,t^{-1}]\right)\oplus \mathbb C\omega_{0},\qquad \widehat{\mathfrak g}^1=\left(\mathfrak g\otimes \mathbb C[t,t^{-1}]u\right)\oplus \mathbb C\omega_{-4}\oplus \mathbb C\omega_{-3}\oplus \mathbb C\omega_{-2}\oplus \mathbb C\omega_{-1}
$$
with bracket
\begin{align*}
[x\otimes t^i,y\otimes t^j]&=[x,y]\otimes t^{i+j}+\delta_{i+j,0}j(x,y)\omega_0, \\ \\
[x\otimes t^{i-1}u,y\otimes t^{j-1}u]&=[x,y]\otimes (t^{i+j+2}-2ct^{i+j}+t^{i+j-2}) \\
 &\hskip 40pt+\left(\delta_{i+j,-2}(j+1) -2cj\delta_{i+j,0} +(j-1)\delta_{i+j,2}\right)(x,y)\omega_0, \\ \\
[x\otimes t^{i-1}u,y\otimes t^{j}]&=[x,y]u\otimes t^{i+j-1}+ j(x,y)\psi_{ij}(c).
\end{align*}

\end{thm}

\begin{proof}
The first two equalities follow from \propref{cocyclecalc}.  For the last one we first observe that for $k=i+j-2\neq -3$,
\begin{align*}
j\omega_{ij}(c)=\overline{ t^{i-1}u\,d( t^{j})}&=j\overline{t^{i+j-2}u\,dt} \\
&=j\left(\frac{-2(k-3)\overline{t^{k-4}u\,dt} +4kc\overline{t^{k-2}u\,dt}}{6+2k}\right),
\end{align*}
where the last equality is derived from \eqnref{recursionreln1}.
Then by setting $k=0,1,2,3,4,5$ in \eqnref{recursionreln1}
\begin{equation*}
(6+2k)\overline{t^{k}u\,dt}   
=-2(k-3)\overline{t^{k-4}u\,dt} +4kc\overline{t^{k-2}u\,dt}.
\end{equation*}
gives us 
\begin{align*}
6\overline{u\,dt}  &=6\overline{t^{-4}u\,dt}  , \\
 8\overline{tu\,dt}   &=4\overline{t^{-3}u\,dt} +4c\overline{t^{-1}u\,dt}, \\
  10\overline{t^{2}u\,dt}   
&=2\overline{t^{-2}u\,dt} +8c\overline{ u\,dt}, \\ 
12\overline{t^{3}u\,dt}   
&=12c\overline{tu\,dt} ,\\
14\overline{t^{4}u\,dt}   
&=-2\overline{u\,dt} +16c\overline{t^{2}u\,dt}, \\
16\overline{t^{5}u\,dt}   
&=-4\overline{tu\,dt} +20c\overline{t^{3}u\,dt}, \\
(6+2k)\overline{t^{k}u\,dt}   
&=-2(k-3)\overline{t^{k-4}u\,dt} +4kc\overline{t^{k-2}u\,dt} .
\end{align*}
Hence when $i+j-2=k=0,1,2,3,4,5$
\begin{align*}
\overline{u\,dt}  &=\omega_{-4} , \\
\overline{tu\,dt}   &=\frac{1}{2}\left(\omega_{-3} +c\omega_{-1}\right), \\
 \overline{t^{2}u\,dt}   
&=\frac{1}{5} \omega_{-2} +\frac{4c}{5}\omega_{-4}, \\ 
\overline{t^{3}u\,dt}   
&=\frac{c}{2}\left(\omega_{-3} +c\omega_{-1}\right), \\
 \overline{t^{4}u\,dt}   
&=-\frac{1}{7}\overline{u\,dt} +\frac{8}{7}c\overline{t^{2}u\,dt}=-\frac{1}{7}\omega_{-4} +\frac{8}{7}c\left(\frac{1}{5} \omega_{-2} +\frac{4c}{5}\omega_{-4}\right)\\
&=\left(\frac{32c^2-5}{35}\right)\omega_{-4} +\frac{8}{35}c \omega_{-2} , \\ 
 \overline{t^{5}u\,dt}   
&=-\frac{1}{8} \left(\omega_{-3} +c\omega_{-1}\right)+\frac{5c^2}{8}\left(\omega_{-3} +c\omega_{-1}\right) \\
&=\frac{5c^2-1}{8} \left(\omega_{-3} +c\omega_{-1}\right),
\\
\overline{t^{k}u\,dt}    
&=\frac{-2(k-3)\overline{t^{k-4}u\,dt} +4kc\overline{t^{k-2}u\,dt}}{6+2k} .
\end{align*}
Thus by induction using the last equation above for $i+j-2=k=2n-3\geq 1$, $n\in\mathbb Z$, we have 
\begin{align}\label{oddcase}
\omega_{ij}(c)
&=P_{-3,i+j-2}(c)\left(\omega_{-3} +c\omega_{-1}\right),
\end{align}
and for $i+j-2=k=2n-2\geq 0$, $n\in\mathbb Z$, we have 
\begin{align}\label{evencase}
\omega_{ij}(c)
&=P_{-4,i+j-2}(c) \omega_{-4} +P_{-2,i+j-2}(c)\omega_{-2}.
\end{align}
Applying $\sigma$ to \eqnref{oddcase} for $i+j-2=k=2n-3\geq 1$ to obtain 
\begin{align*}
j\sigma(\omega_{ij}(c))=\overline{ t^{-i+1}u\,d( t^{-j})}&=-j\overline{t^{-i-j-2}u\,dt} \\
&=j\sigma\left(P_{-3,i+j-2}(c)\left(\omega_{-3} +c\omega_{-1}\right) \right) \\ 
&=-jP_{-3,i+j-2}(c)\left(\omega_{-1} +c\omega_{-3}\right).
\end{align*}
Hence for $i+j-2=2n-3\geq 1$
\begin{align*}
\omega_{-i,-j}(c)=\overline{t^{-i-j-2}u\,dt} 
&=P_{-3,i+j-2}(c)\left(\omega_{-1} +c\omega_{-3}\right).
\end{align*}
Setting $i'=-i$ and $j'=-j$ we get for $i'+j'-2=-k-4=-2n+3\leq -5$
\begin{align*}
\omega_{i'j'}(c)=\overline{t^{i'+j'-2}u\,dt} 
&=P_{-3,|i'+j'|-2}(c)\left(\omega_{-1} +c\omega_{-3}\right).
\end{align*}

Similarly if  we apply $\sigma$ to \eqnref{evencase} for $i+j=2n\geq 2$, $n\in\mathbb Z$, we obtain 
\begin{align*}
j\sigma(\omega_{ij}(c))=\overline{ t^{-i+1}u\,d( t^{-j})}&=-j\overline{t^{-i-j-2}u\,dt} \\
&=j\sigma\left(P_{-4,i+j-2}(c) \omega_{-4} +P_{-2,i+j-2}(c)\omega_{-2} \right) \\ 
&=-j\left(P_{-4,i+j-2}(c) \omega_{-4} +P_{-2,i+j-2}(c)\omega_{-2}\right)
\end{align*}
Hence for $i+j=2n\geq 2$
\begin{align*}
\omega_{-i,-j}(c)=\overline{t^{-i-j-2}u\,dt} 
&=P_{-4,i+j-2}(c) \omega_{-4} +P_{-2,i+j-2}(c)\omega_{-2}.
\end{align*}
Setting $i'=-i$ and $j'=-j$ we get for $i'+j'=-2n\leq  -2$
\begin{align*}
\omega_{i'j'}(c)=\overline{t^{i'+j'-2}u\,dt} 
&=P_{-4,|i'+j'|-2}(c) \omega_{-4} +P_{-2,|i'+j'|-2}(c)\omega_{-2}.
\end{align*}
\end{proof}
One might want to compare the above theorem with the results that M.  Bremner obtained for the elliptic and four point affine Lie algebra cases (\cite[Theorem 4.6]{MR1303073} and \cite[Theorem 3.6]{MR1249871} respectively).

\def\cprime{$'$} \def\cprime{$'$} \def\cprime{$'$}

\end{document}